\documentclass{amsart}
\usepackage{amssymb}
\usepackage[foot]{amsaddr}
\usepackage[utf8]{inputenc}
\usepackage[T1]{fontenc}
\usepackage[english]{babel}
\usepackage{mathtools}
\usepackage{csquotes}
\usepackage{comment}
\usepackage{xargs}
\usepackage{graphicx, color}
\setkeys{Gin}{draft=false}
\usepackage[figurewithin=none]{caption,subcaption}
\graphicspath{{./pictures/}} %Uso: \graphicspath{{RelativePath}}
\usepackage{enumitem}
\setlist[enumerate,1]{label={(\roman*)}}

\usepackage[pagewise,modulo,mathlines]{lineno}
\usepackage[notcite,notref,color,final]{showkeys}

\usepackage{hyperref}
\hypersetup{
  hidelinks,
  colorlinks  = true,    % Colours links instead of ugly boxes
  urlcolor    = blue,    % Colour for external hyperlinks
  linkcolor   = blue,    % Colour of internal links
  citecolor   = blue,      % Colour of citations
  pdfsubject  = {Maniplexes},%
  pdftitle    = {Every finite premaniplex of rank 4 is the symmetry type graph of a finite maniplex},  %
}

\usepackage[capitalise,noabbrev, nameinlink]{cleveref}
\usepackage{tikz-cd}
\tikzcdset{%
row sep=tiny,%
arrows={shorten=-4pt}
}
\usetikzlibrary{shapes.geometric}
\usetikzlibrary{decorations.markings}
\usetikzlibrary{calc}

\usepackage[loadshadowlibrary,shadow, colorinlistoftodos,textsize=tiny,obeyFinal]{todonotes}
\setuptodonotes{fancyline, backgroundcolor=gray!10,bordercolor=gray}

\newcommandx{\inline}[2][1=]{\todo[inline, #1]{#2}}

\makeatletter
  \providecommand\@dotsep{5}
\makeatother

\theoremstyle{plain}
\newtheorem{theorem}{Theorem}[section]

\newtheorem{lemma}[theorem]{Lemma}
\theoremstyle{definition}

\newtheorem{problem}[theorem]{Problem}
\newtheorem{question}[theorem]{Question}

\theoremstyle{remark}
\newtheorem{remark}[theorem]{Remark}

\newcommand{\V}{\mathrm{V}}

\newcommand{\ZZ}{\mathbb{Z}}
\newcommand{\U}{\mathcal U}
\renewcommand{\P}{\mathcal P}
\newcommand{\N}{\mathcal N}
\renewcommand{\H}{\mathcal H}

\newcommand{\G}{\mathcal G}
\newcommand{\No}{\mathbf N}
\newcommand{\Aut}{\mathrm{Aut}}
\newcommand{\id}{{\mathrm{id}}}
\DeclareMathOperator{\STG}{STG}
\newcommand{\stg}{\STG}
\DeclareMathOperator{\CT}{CT}
\DeclareMathOperator{\beg}{beg}
\DeclareMathOperator{\inv}{inv}

\definecolor{myblue}{RGB}{0, 70, 180}
\definecolor{myred}{RGB}{180, 40, 40}
\tikzset{
  vertex/.style={
    circle,
    draw,
    fill=white,
    inner sep=2pt,
    minimum size=6pt
  }
}

\usepackage[%
sorting=nyt, %
backend=bibtex, %
style=numeric,%
url=false,%
eprint=false,%
giveninits=true,
maxbibnames=9,%
isbn=false,%
sortcites = true,%
citestyle=numeric-comp,
]{biblatex}

\title{Premaniplexes of rank $3$ and $4$ as symmetry type graphs of maniplexes}
\author{Maru\v{s}a Lek\v{s}e}
\address{Institute of Mathematics, Physics and Mechanics, Jadranska 19, SI-1000 Ljubljana, Slovenia and Faculty of Mathematics and Physics, University of Ljubljana, SI-1000 Ljubljana, Slovenia}
\email{marusa.lekse@imfm.si}
\begin{document}

\begin{abstract}
We show that every finite premaniplex of rank $3$ or $4$ is the symmetry type graph of
a finite maniplex, settling the finite rank $3$ and $4$ case of the maniplex version of 
the symmetry type graph problem. The proof uses the fact that the universal string
 Coxeter groups of
rank $3$ and $4$ are amalgams (of finite groups), hence they act on a tree, 
which allows us to use a
lifting theorem of Potočnik and Spiga.
\end{abstract}

\maketitle
%\marusa{Abstract could be better}

\section{Introduction}
A \emph{maniplex} of rank $n$ is a connected simple $n$-valent graph whose edges 
are properly coloured with the colours
$0,1,\dots,n-1$, and in which edges
of colours $i$ and $j$ form alternating $4$-cycles whenever $|i-j| \ge 2$. 
They were introduced by Wilson~\cite{Wilson2012} as a generalisation of maps
and of abstract polytopes. The flag graphs of abstract polytopes are maniplexes, 
while not every maniplex is the flag graph
of a polytope. Besides being an interesting object to study on their own,
maniplexes are also a useful tool to study abstract polytopes using graph theory. 

The quotient of $M$ by its automorphism group is called the \emph{symmetry type graph}
$\STG(M) = M/\Aut(M)$. They were defined in \cite{CunninghamRioHubardToledo2015}.
The quotient of a maniplex of rank $n$ by a group of automorphisms is always a
\emph{premaniplex} of rank $n$: a connected graph whose edges and semiedges are 
coloured with $n$
colours, one of each colour at each vertex, and in which $4$-walks of alternating colours
 $i$ and $j$ are closed whenever $|i-j|\ge2$. In particular, every maniplex is a premaniplex, 
 and every symmetry type graph 
 of a maniplex of rank $n$
is a premaniplex of rank $n$. Whether every premaniplex can be obtained 
as the symmetry type graph of some maniplex is
 an open question
(see~\cite{CunninghamPellicer2018} for a list of open problems on
polytopes and maniplexes).

\begin{problem}[{\cite[Problem 12]{CunninghamPellicer2018}, \cite{HubardMochan2023}}]\label{prob:main}
	Given a premaniplex $X$, does there exist a maniplex $M$ with $\STG(M)\cong X$?
  Can $M$ be chosen to be polytopal (that is, can $M$ be the flag graph of a polytope)?
\end{problem}

For premaniplexes with two vertices, the maniplex version was resolved by Pellicer,
Poto\v{c}nik and Toledo~\cite{PellicerPotocnikToledo2019}.
The polytopality of the maniplexes that they constructed was later analysed by
Moch\'{a}n~\cite{Mochan2024}. 
Progress has also been made for rank $3$. The author and Toledo show 
in \cite{LekseToledo} that every (polytopal) map 
(a premaniplex of rank $3$ with no semiedges and no parallel $0$ and $2$ edges)
can be obtained as the symmetry type graph of a (polytopal) map. This settles a large part of 
rank $3$.
%Even more,Potočnik 
%shows in \cite{Potocnik} that every premaniplex of rank $3$ can be obtained as the symmetry type graph of a maniplex, 
%where it is even possible to control the type of the maniplex that covers the original maniplex.
The aim of this note is to study the maniplex version for rank $3$ and $4$.
We will show that for any finite premaniplex $X$ of rank $3$ or $4$ and a group $G\le \Aut(X)$ there
exists a finite maniplex $M$ that is a regular cover of $X$, and such that 
$\Aut(M)$ is precisely the lift of $G$.

\begin{theorem}\label{th:main}
	Every finite premaniplex of rank $3$ or $4$ is the symmetry type graph of a finite maniplex.
\end{theorem}

%%
%The strategy of showing this is viewing a finite maniplex $M$ of rank $4$ in the following way.
%Let $W$ be the universal string coxeter group of rank $4$, and $\mathcal U$ the universal
%maniplex of rank $4$. Then $M \cong \mathcal U/\N$ for some $\N \le W$.
%Beside on the universal maniplex, $W$ also acts on an infinite tree $T$, and we can 
%instead study the quotient
%$M' := T/\N$. It is not the original maniplex, but a finite $(4,2)$-valent graph. 
%This enables us to use a theorem of Poto\v{c}nik and
%Spiga~\cite{PotocnikSpiga2019} to find an appropriate finite cover $T/\P$ of $M'$.
%Finally, we look at the quotient $\mathcal U/ \P$, which is the cover of $M$ that we aim to construct.

The proof proceeds as follows. 
Let $X$ be a finite premaniplex of rank $3$ or $4$. 
We first find a finite maniplex $M$ and a group $G\le\Aut(M)$ with 
$M / G \cong X$. After that it
suffices to construct a finite cover of $M$ whose full automorphism group is precisely
the lift of $G$, as this implies that $\stg(M) = X$.

Let $W$ be the universal string Coxeter group of rank $3$ or $4$ and let $\U$ be the universal
maniplex of rank $3$ or $4$, so that $M \cong \U/ \N$ for a subgroup $\N\le W$ of
finite index. Besides
acting on $\U$, the group $W$ also acts on an infinite tree $T$, and we
can instead study the quotient $\Gamma:=T / \N$. 
This is not the original maniplex,
it is a finite graph whose vertices have valency $2$
or $4$. This enables us to use a theorem of Potočnik and 
Spiga \cite{PotocnikSpiga2019}, which
gives us a subgroup $\P\le\N$ (such that exactly the prescribed
automorphisms lift from $T/\N$ to $T/ \P$). The same $\P$ then gives a
cover $\U/\P$ of $M$ that we aimed to construct.

\section{Preliminaries}
All group actions will be on the right. We will use $G/H$ to denote the right cosets space
for a subgroup $H$ of $G$.
A \emph{graph} is
a $4$-tuple $\Gamma=(\mathrm{D},\V;\beg,\inv)$ of darts and vertices, an edge is a pair
$\{x,x^{-1}\}$ of mutually inverse darts, a \emph{semiedge} is an edge with $x=x^{-1}$,
and the \emph{valency} of a vertex is the number of darts with that tail. Since maniplexes
and premaniplexes are undirected graphs, we will lighten the notation by giving the set of
vertices (which, as is standard in maniplexes, are called \emph{flags}) and then the
sets of $i$-edges for $i \in \{0, \ldots, n-1\}$.
%
%\begin{definition}\label{def:premaniplex}
	%A \emph{premaniplex of rank $4$} is a connected graph $X$ together with a colouring of
	%its edges and semiedges by the four colours $0,1,2,3$, such that
	%\begin{enumerate}
	%	\item for each colour $i$, every vertex of $X$ is the tail of exactly one dart
	%		of colour $i$ (in particular $X$ has no loops, and every vertex has valency $4$); and
	%	\item for each pair $\{i,j\}$ with $|i-j|\ge2$, every walk of length $4$ with alternating colours $i$ and $j$ is closed.
	%\end{enumerate}
	%A premaniplex is a \emph{maniplex of rank $4$} if it contains no semiedges and no parallel
	%edges. The walks in (ii) are then $4$-cycles.
%\end{definition}
 An \emph{automorphism} of a (pre)maniplex is
a colour-preserving automorphism of the graph. It always acts semi-regularly.

A \emph{covering projection} of graphs is a surjective graph homomorphism
$\wp:\tilde\Gamma\rightarrow\Gamma$ that satisfies the local bijection condition: 
for every vertex $\tilde u$ of
$\tilde\Gamma$, $\wp$ maps the darts with tail $\tilde u$ bijectively onto the darts with tail
$\wp(\tilde u)$. We call $\tilde\Gamma$ a \emph{cover} of $\Gamma$
%, and the \emph{fibre} over a vertex $u$ of $\Gamma$ is the set $\wp^{-1}(u)$. 
The group of \emph{covering transformations} $\CT(\wp)$ is the group of those automorphisms 
of $\tilde\Gamma$ that
preserve every fibre of the projection setwise. If $\CT(\wp)$ acts regularly
on each fibre, then we say that $\wp$ is a
\emph{regular} covering projection and $\tilde\Gamma$ is a \emph{regular cover} of
$\Gamma$. In this case $\Gamma\cong\tilde\Gamma/\CT(\wp)$. An automorphism $\alpha$ of
$\Gamma$ \emph{lifts} along $\wp$ if there is an automorphism $\tilde\alpha$ of
$\tilde\Gamma$ with $\wp\tilde\alpha=\alpha\wp$. If every automorphism in a subgroup
$G\le\Aut(\Gamma)$ lifts along $\wp$, then the set of all lifts of all elements of $G$ is a
subgroup of $\Aut(\tilde\Gamma)$, called the \emph{lift} of $G$.
For covers of (pre)maniplexes we require
in addition that $\wp$ preserve the colours of darts.

Regular covers can be given by voltages (\cite{Gross1974, GrossTucker1977, Malnic1998, MalnicNedelaSkoviera2000}). Given a graph $\Gamma=(\mathrm{D},\V;\beg,\inv)$  and a group
$G$, a \emph{voltage assignment} on $\Gamma$ is a map $\zeta: \mathrm{D} \rightarrow G$ 
such that $\zeta(\inv(x))=\zeta(x)^{-1}$ for every dart $x$. The \emph{derived graph} $\tilde \Gamma$ has
vertex set $\V(\Gamma)\times G$ and dart set $\mathrm{D}(\Gamma)\times G$, where the dart
$(x,g)$ has tail $(\beg(x),g)$ and inverse $(x^{-1},\zeta(x)g)$. If $\Gamma$ is a
(pre)maniplex we give $(x,g)$ the colour of $x$. The projection $\wp:(u,g)\mapsto u$ is a covering projection,
and if $\tilde \Gamma$ is connected, then it is regular with $\CT(\wp)\cong G$. Every regular cover can be given as a voltage assignment.
If $\tilde \Gamma$ is not connected,
then the projection $\wp|_{\tilde\Gamma_1}:\tilde\Gamma_1 \rightarrow \Gamma$
induced on a connected component $\tilde\Gamma_1$ of $\tilde\Gamma$ is a regular covering projection.
The voltage of a path is defined as the product of voltages of its underlying darts.
An automorphism $g$ of $\Gamma$ lifts along $\wp$ (or $\wp|_{\tilde\Gamma_1}$) if and only if for every closed walk $W$ in $\Gamma$ that has trivial voltage,
the walk $W^g$ also has trivial voltage.

\subsection{Maniplexes as subgroups of string Coxeter groups}\label{sec:quotients} 

The universal string Coxeter group of rank $n$ is
\[
	W_n = \langle\, r_0,\dots,r_{n-1} \mid r_i^2,\ (r_ir_j)^2 \ \text{for } |i-j|\ge2 \,\rangle .
\]
Note that for $n=3$
\[
	W_3 \cong \langle r_0,r_2\rangle *\langle r_1\rangle,
\]
and that for $n=4$
\[
	W_4 \cong \langle r_0,r_2\rangle *_{\langle r_0\rangle} \langle r_0,r_3\rangle
		*_{\langle r_3\rangle} \langle r_1,r_3\rangle .
\]

The elements $r_0, \ldots r_{n-1}$ are called the \emph{standard generators} of $W_n$.
The \emph{universal maniplex} $\U_n$ of rank $n$ is the Cayley graph of $W_n$ with respect
to the generating set $\{r_0,\dots,r_{n-1}\}$ (its vertices are the elements of $W_n$, and
$w$ is connected to $r_iw$ by an edge of colour $i$). 
The group $W_n$ acts on $\U_n$ by multiplication on the
right, and $\Aut(\U_n)=W_n$.

Maniplexes can be studied as subgroups of $W_n$.
For a subgroup $\N\le W_n$ let $\U_n/\N$ be the quotient of $\U_n$ by $\N$ 
(that is, the graph whose vertices are the
orbits of $\N$ on $\U_n$, in which $w\N$ is joined to $r_iw\N$ by an edge of
colour $i$; if $w\N=r_iw\N$, this is a semiedge). Note that, because $\N$ acts on the right, these orbits are the left cosets $w\N$.
The lemmas in the rest of this section are folklore, see for example
 \cite{HubardToledo2023,JonesSingerman1978,HubardMochanMontero2025,Orbanic2007,Malnic1998}.
We include some proofs for the sake of completeness.

\begin{lemma}\label{lem:quotient}
	Let $X$ be a premaniplex of rank $n$. Then the following hold:
	\begin{enumerate}
		\item $X\cong\U_n/\N$ for some subgroup $\N\le W_n$,
		\item $\U_n/\N$ is a maniplex if and only if $\N$ contains no
			conjugate of $r_i$ and no conjugate of $r_ir_j$, for all $i\ne j$.
	\end{enumerate}
\end{lemma}

%\marusa{cite something also for this. ALSO be careful, this is for N without restrictions,  for premaniplexes}

\begin{lemma}\label{lem:aut}
	Let $\N\le W_n$. Then every automorphism of $\U_n/\N$ lifts to an automorphism of $\U_n$, 
  the lift of $\Aut(\U_n/\N)$ is equal to $\No_{W_n}(\N)$. More precisely,
    $\No_{W_n}(\N)$ acts on $\U_n/\N$ by $\pi(g): w\N\mapsto wg\N$. 
    This action has kernel $\N$, and the induced action of  $\No_{W_n}(\N)/\N$ on
     $\U_n/\N$ is permutationally isomorphic to
    $\Aut(\U_n/\N)$.
\end{lemma}

In the following we will identify $\Aut(\U_n/\N)$ with $\No_{W_n}(\N)/\N$.

\begin{lemma}\label{lem:cover}
	Let $\P\trianglelefteq\N\le W_n$ and let
	$\wp:\U_n/\P\rightarrow\U_n/\N$ be the map defined by $w\P\mapsto w\N$. Then
	$\wp$ is a regular covering projection with $\CT(\wp)=\N/\P$.
	The largest subgroup of $\Aut(\U_n/\N)$ that lifts along $\wp$ is
			\[
				\bigl(\No_{W_n}(\N)\cap\No_{W_n}(\P)\bigr)/\N.
			\]
	Its lift is $\bigl(\No_{W_n}(\N)\cap\No_{W_n}(\P)\bigr)/\P$.
	\end{lemma}

\begin{proof}
It is straightforward to check that $\wp$ is a regular covering projection with the group of covering
transformations $\N/\P$.

Assume first that $g\N \in (\No_{W_n}(\N)\cap\No_{W_n}(\P))/\N$. Then by \cref{lem:aut}, $g\N$ is an automorphism of $\U_n/\N$
and $g\P$ is an automorphism of $\U_n/\P$. Even more, by the definition of $g\N$ and $g\P$, the former is the projection
of the latter along $\wp$.

Conversely let $g\N$ be an automorphism of $\U_n/\N$ that lifts to an automorphism $h\P$ of $\U_n/\P$. 
By \cref{lem:aut},
it follows that $g \in \No_{W_n}(\N)$ and that $h \in \No_{W_n}(\P)$. Since $h\P$ is a lift of $g\N$ along
$\wp$, it follows that $g\N$ must map $\wp(\P) = \N$ into $\wp(g\P) = g\N$, or 
equivalently that $g\N = h\N$. Therefore $h \in \No_{W_n}(\N)$. Hence $g\N = h\N \in (\No_{W_n}(\N)\cap\No_{W_n}(\P))/\N$,
which proves the first claim.
Note that we showed that $h\P$ is a lift of $h\N$ for every $h \in \No_{W_n}(\N)\cap\No_{W_n}(\P)$. This shows that the lift of $(\No_{W_n}(\N)\cap\No_{W_n}(\P))/\N$ is
$(\No_{W_n}(\N)\cap\No_{W_n}(\P))/\P$.
%
%
    %Let $\pi_{\P}$ and $\pi_{\N}$ denote the maps defined in \cref{lem:aut} for $\P$ and $\N$. 
    %An automorphism of $\U_n/\N$ lifts if and only if it is the projection of an automorphism of 
    %$\U_n/\P$. By \cref{lem:aut} an automorphism of $\U_n/\P$ is of the form $\pi_{\P}(h)$ for some
	  %%$h\in\No_{W_n}(\P)$.
    %It projects to $\U_n/\N$ if and only if
%
%It projects to an automorphism $\U_n/\N$ if and only if its projection 
%		$\omega\N \mapsto \omega h \N$ is well de, 
%		which, by \cref{lem:aut}, is if and only if $h\in\No_{W_n}(\N)$.
%		 Hence for $h\in\No_{W_n}(\P)$, the automorphism $\pi_{\P}(h)$ projects to an automorphism of 
%    $\U_n/\N$ if and only if $h$ normalizes $\N$. The automorphisms of $\U_n/\N$ that lift along $\wp$ 
 %    are therefore exactly those of the form $\pi_{\N}(h)$ with
%	$h\in\No_{W_n}(\N)\cap\No_{W_n}(\P)$, and since $\pi_{\N}$ has kernel $\N$ this gives us the expression
 %  in the statement of the lemma.
\end{proof}

\begin{lemma}\label{lem:cover2}
Let $X$ be a finite premaniplex of rank $n$. Then there exist a finite maniplex $M$ and a
regular covering projection $\wp\colon M\rightarrow X$ along which every automorphism of $X$
lifts. 
\end{lemma}

\begin{proof}
	Let $G=\ZZ_2^n$ with standard basis $e_0,\dots,e_{n-1}$, and let $\zeta$ assign to every
	dart of colour $i$ of $X$ the voltage $e_i$. Since every $e_i$ is an involution, this is
	a voltage assignment. Let $M$ 
be a connected component of the derived graph. Then $M$ is a regular cover of $X$. We claim that $M$ is a maniplex. Since $M$ is a cover of 
  $X$, it is a properly $n$-edge coloured graph. It contains no parallel edges,
  since such edges have different colour, and edges of different colours have
  voltage whose product is not $0$. Since no dart has trivial voltage, $M$ contains no semiedges. Since every $i,j$-walk $W$ for $|i-j|\ge2$ in $X$ of length $4$
  is closed and has voltage $0$, its lift $\tilde W$ is also closed in $M$. Additionally, since
  no proper non-empty subwalk of $W$ has voltage $0$, it follows that $\tilde W$ is a cycle. This proves that $M$ is a maniplex.

	Finally, let $g$ be an automorphism of $X$. As $g$ preserves the colours of
	darts and the voltage of a dart depends only on its colour, $\zeta(W^g)=\zeta(W)$, for
	every walk $W$ of $X$. Therefore $g$ lifts to an automorphism of $M$.
\end{proof}

\subsection{Bass--Serre trees for $W_3$ and $W_4$}\label{sec:tree}

In this section we define trees $T_3$ and $T_4$, which we will need later, and recall some facts about them from the Bass--Serre theory.
See~\cite{Serre1980} for more (note that we use right actions, and
\cite{Serre1980} uses left actions). 

For $n=3$ we have $W_3=\langle r_0,r_2\rangle*\langle r_1\rangle$, so we may define a tree $T_3$ with
the vertex set equal
to the disjoint union of $W_3/\langle r_0,r_2\rangle$ and $W_3/\langle r_1\rangle$, its
edge set is $W_3$, and the edge $w$ joins $\langle r_0,r_2\rangle w$ to
$\langle r_1\rangle w$. The vertices in $W_3/\langle r_0,r_2\rangle$ have valency $4$ and
those in $W_3/\langle r_1\rangle$ have valency $2$. The group $W_3$ acts on $T_3$ by right
multiplication. The vertex stabilisers are the
conjugates of $\langle r_0,r_2\rangle$ and of $\langle r_1\rangle$, and the edge
stabilisers are trivial. There are two orbits of vertices and one orbit of edges for the action of $W$.

For $n=4$ we have $W_4\cong\langle r_0,r_2\rangle *_{\langle r_0\rangle} \langle r_0,r_3\rangle
*_{\langle r_3\rangle} \langle r_1,r_3\rangle$. Let $T_4$ be the tree with the vertex set equal to the
disjoint union of the coset spaces
\[
	W_4/\langle r_0,r_2\rangle, \qquad W_4/\langle r_0, r_3\rangle, \qquad W_4/\langle r_1,r_3\rangle,
\]
its edge set is the disjoint union of $W_4/\langle r_0\rangle$ and
$W_4/\langle r_3\rangle$, and $\langle r_0\rangle w$ joins
$\langle r_0,r_2\rangle w$ to $\langle r_0, r_3\rangle w$, while $\langle r_3\rangle w$ joins
$\langle r_0, r_3\rangle w$ to $\langle r_1, r_3\rangle w$.
The vertices in $W_4/\langle r_0,r_3\rangle$ have valency $4$, and all the other vertices
have valency $2$. Similarly, the group $W_4$ acts on $T_4$ by right multiplication. The vertex stabilisers are the
conjugates of $\langle r_0,r_2 \rangle$,
$\langle r_0,r_3\rangle$ and $\langle r_1,r_3\rangle$, and the edge stabilisers are the
conjugates of $\langle r_0\rangle$ and of $\langle r_3\rangle$. There are three orbits of
vertices and two orbits of edges for the action of $W$.

\section{Proof of the main theorem}
We start by recalling the lemma that is the main ingredient of our proof.

\begin{theorem}[{\cite[Theorem 5]{PotocnikSpiga2019}}]
\label{thm:lifting5}
Let $p$ be an odd prime, let ${T}$ be an infinite tree, let
$\G\le\operatorname{Aut}({T})$, let $\N$ be a
non-identity normal subgroup of $\G$ of finite index such that
$\N_x=1$ for every vertex and for every edge $x$ of ${T}$,
and let $\H=\No_{\operatorname{Aut}({T})}(\N)$.
If $\H/\N$ acts faithfully on
$\mathrm{H}_1({T}/\N;\mathbb{Z})$, then there exists a normal
subgroup $\P$ of $\N$ of finite index such that
$\No_{\H}(\P)=\G$ and $\N/\P$
is a $p$-group.
\end{theorem}

The hypothesis about the first homology group is not difficult to satisfy. The next
lemma is proved in~\cite[Proposition 5]{EstelyiKarabasMednykhNedela2021} for simple
graphs, and the general case follows by subdividing every edge twice. We include a
proof for the sake of completeness.

\begin{lemma}\label{rem:H1}
	Let $\Gamma$ be a finite connected graph without semiedges in which every vertex has
	valency at least $2$ and which is not a cycle. Then $\Aut(\Gamma)$ acts faithfully on
	$\mathrm{H}_1(\Gamma;\ZZ)$.
\end{lemma}

\begin{proof}
  It is enough to show that there
  exists no nonidentity $g\in\Aut(\Gamma)$ that fixes every cycle
  setwise, as well as their direction. Suppose that such a $g$ exists.
  
  %Note that because of the assumption
  %on degrees of vertices, $\Gamma$ contains a cycle. 
  Note also that $g$ cannot fix all vertices and only move some dart $x$, since then either $x$ is a loop, or the edges of $x$ and $x^g$, have the same
  endvertices. In both case $g$ does not fix all cycles together with their directions.

  We first aim to show that there exists at least one vertex that is moved by $g$ and lies on a
  cycle. Suppose that there exists no such vertex. Let $u$ be a vertex that is moved by $g$ and lies on
  no cycles. Observe that it is a cut-vertex, and that each of the at least two connected
  components of $\Gamma - u$ contains a cycle. Choose vertices $a$ and $b$ lying on cycles in two different components
  of $\Gamma-u$. Then $u$ lies on every path from $a$ to $b$. By our assumption $a$ to $b$ are fixed by $g$.
The vertex $u^g$ also lies on every path from $a$ to $b$, and $u$ and $u^g$ are at the same distance from $a$.
Therefore they
  lie at the same position on any shortest path from $a$ to $b$. This implies that $u^g=u$, which is a contradiction.
  We choose $v$ to be
  a vertex that is moved by $g$ and lies on a cycle.

  Now let $C$ be the cycle containing $v$. Since $g$ fixes the cycles
  setwise, and their directions, $C$ also contains $v^g$, and $g$ induces
  a nontrivial rotation on $C$. Let $\Gamma \setminus C$ denote the graph obtained from $\Gamma$ by
  deleting the edges of $C$. Each of its components contains a vertex of $C$.
  Now we separate two options. Either there exist a walk $W$
  in $\Gamma \setminus C$ connecting two distinct vertices in $C$, or each component of
  $\Gamma \setminus C$ contains exactly one vertex of $C$. In the first case, there exists a cycle $C'$ containing $W$
  and an arc of $C$. Since $g$ also fixes $C'$ setwise, and preserves its direction,
  it fixes that arc, and hence its endvertices, contradicting that it acts on $C$ as a nontrivial rotation.
  In the second case, each component of
  $\Gamma \setminus C$
  contains a cycle whenever it contains at least two vertices. Such a component exists because
  $\Gamma$ is not a cycle. The
  components are permuted by $g$, and none of them is fixed, hence $g$ moves a cycle.

  In both cases we arrive to a contradiction, which completes the proof.
\end{proof}

We use \cref{thm:lifting5} in the following form.

\begin{theorem}\label{thm:lifting5'}
Let $p$ be an odd prime, let $W\in \{W_3, W_4\}$, let $\G\le W$ and let $\N$ be a non-identity
normal subgroup of $\G$ of finite index in $W$ that contains none of the
elements $r_i$, $r_ir_j$ for $i \ne j$,
 or their conjugates. 
 Let
$\H'=\No_{W}( \N)$. Then there exists a normal
subgroup $\P$ of $\N$ of finite index such that
$\No_{\H'}(\P)=\G$ and $\N/\P$ is a
$p$-group.
\end{theorem}

\begin{proof}
	Let $T$ be the Bass-Serre tree of $W$ that is described in \cref{sec:tree}. In both cases it
	is an infinite tree without semiedges. The action of $W$ on $T$ is faithful, since its kernel is
	contained in the intersection of the vertex stabilisers, which is trivial in both cases. In particular,
	$W\le\Aut(T)$. Let $\H$ be the
	normaliser of $\N$ in $\Aut(T)$. We will check the hypotheses of \cref{thm:lifting5}
	for the tree $T$ and the group $\N\trianglelefteq\G$.

	The index of $\N$ in $\G$ is finite by assumption. Every non-trivial element of a vertex
	stabiliser of $T$ is a conjugate of some $r_i$ or of some $r_ir_j$ with $i\ne j$, so the
	vertex stabilisers, and hence also the edge stabilisers, intersect $\N$ trivially.
	Therefore $\N_x=1$ for every vertex and for every edge $x$ of $T$.

	It remains to check that $\H / \N$ acts faithfully on
	$\mathrm{H}_1(T/ \N; \ZZ)$. Since $\N$ acts freely on the
	vertices of $T$, and hence on its darts, the valencies are preserved. 
  Therefore every vertex of
	$T/\N$ has valency $2$ or $4$ and at least one has valency $4$.
   Since there are no edges between vertices
   in the same $W$-orbit on $T$, $T/\N$ contains no semiedges. Note that it is connected. 
   Since $\N$ has finite
    index in $W$,
	which has finitely many orbits on the vertices and on the edges of $T$,
	 $T/ \N$ is finite.
	The group $\H/\N$ is a subgroup of the full automorphism group of the
	graph $T/\N$, hence it suffices to show that $\Aut(T/\N)$ acts
	faithfully on $\mathrm{H}_1(T/ \N ; \ZZ)$, which it does by \cref{rem:H1}.

By \cref{thm:lifting5} there exists a normal subgroup $\P$ of $\N$ of
finite index with $\N/\P$ a $p$-group and
$\No_{\H}(\P)=\G$. Finally note that $\H'=\H\cap W$,
so $\No_{\H'}(\P)=\No_{\H}(\P)\cap W
=\G\cap W$. Since $\G\le W$, it follows that $\No_{\H'}(\P) = \G$.
\end{proof}

In fact for  a finite premaniplex of rank $3$ or $4$, we can prescribe exactly which
subgroup of automorphisms should lift, following \cite{PotocnikSpiga2019}.

\begin{theorem}\label{th:lift}
	Let $X$ be a finite premaniplex of rank $3$ or $4$ and let $G\le\Aut(X)$. Then there
	exist a finite maniplex $M$ and a regular covering projection $\wp\colon M\rightarrow X$
	such that
	such that $\Aut(M)$ is the lift of $G$ along $\wp$.
\end{theorem}

\begin{proof}
    By \cref{lem:cover2}, there exists a cover $\wp': M_1 \rightarrow X$, where $M_1$ is a maniplex, such that every automorphism of $X$ lifts to $M_1$. Let $\hat G = \wp'^{-1}(G)$. We aim to show that the maniplex $M_1$ admits a finite regular cover $\wp'': M \rightarrow M_1$, such that $\Aut(M) = \wp''^{-1}(\hat G)$. 

    Let $W$ and $\U$ be the universal string Coxeter group and the universal maniplex of the
	same rank as $X$.  
    Let $\N$ be a subgroup of $W$ such that 
  $\U/\N \cong M_1$, let $\G$ be the lift of $\hat G$ to $\No_W(\N)$ (so that
  $\G/\N=\hat G$), let $f:=[W:\N]$ be the number of flags of $M_1$, let $p$ be an odd prime with $p>f$, and
	let $\H':=\No_{W}(\N)$. The subgroup $\N$ is normal in $\G$ and of finite index $f$ in
	$W$, and it is non-trivial because $W$ is infinite. Since $\U/\N\cong M_1$ is a maniplex,
	$\N$ contains none of the elements $r_i$, $r_ir_j$ for $i\ne j$, or their conjugates by
	\cref{lem:quotient}. By \cref{thm:lifting5'} there exists a normal
	subgroup $\P$ of $\N$ of finite index with $\No_{W}(\P) \cap \No_{W}(\N)= \No_{\H'}(\P) = \G$ and 
     $\N/\P$ a $p$-group. Let $M = \U/\P$ and $\wp'': M \rightarrow M_1$ be the regular covering projection as defined in \cref{lem:cover}. Note that the largest automorphism group that lifts along $\wp''$ is 
     \[(\No_W(\N) \cap \No_W(\P))/\N =  \G/\N = \hat G.\] 
     
     Next we need to show that that the lift of $\hat G$ along $\wp''$ is the whole automorphism group of $M$, or equivalently that  \[(\No_W(\N) \cap \No_W(\P))/\P =  \No_W(\P)/\P.\]
     Let $Q = \No_W(\P)$. Since $\N$ has index $f$ in $W$, and $f < p$, it has index less than $p$ in $Q$. Recall that $\N/\P$ is a $p$-group, therefore it is a Sylow $p$-subgroup of $Q/\P$. We use this in a similar way to \cite{PotocnikSpiga2019}.
     Let $n_p$ denote the number of
	Sylow $p$-subgroups of $Q/\P$. Then $n_p$ is congruent to $1$ modulo $p$, and it divides
	the index of $\N/ \P$ in $Q/\P$. Since this index is smaller than $p$, it follows that $n_p$ is $1$. Hence $\N/\P$ is the unique Sylow $p$-subgroup of $Q/\P$. In particular, $\N/\P \trianglelefteq Q/\P$.  Hence by the third isomorphism theorem, $N \trianglelefteq \No_W(\P) = Q$. This implies that $\No_W(\P) \le \No_W(\N)$, and 
    hence $\No_W(\N) \cap \No_W(\P))/\P =  \No_W(\P)/\P$.

    We showed that the largest group of automorphisms of $M$ that lifts along $\wp''$ is $\hat G$, and that every automorphism of $M$ projects to $M_1$. Let $\wp := \wp'\circ \wp''$. A composition of covering projections is a covering projection. Since $\wp'{-1}(\id) \le \hat G$, it lifts along $\wp''$. Hence $\CT(\wp)$ acts regularly on the fibres of the projection, and $\wp$ is regular. We have shown that $\Aut(M) = \wp''^{-1}(\hat G) = \wp^{-1}(G)$. This completes the proof.
\end{proof} 

If we take $G=1$, then \cref{th:lift} proves \Cref{th:main}.

\begin{remark} 
  Note that the same method does not work for higher ranks.
  We could generalize \cref{thm:lifting5'} to the statement:
  let $p$ be
	an odd prime, let $W$ be a group that acts faithfully on an infinite tree $T$ that is not
	a line and in which at least one vertex has degree at least $3$, and every vertex has degree at least $2$, with finite vertex- and edge-
	stabilisers and with finitely many orbits on the vertices and on the edges. Let
	$\G\le W$, let $\N$ be a torsion-free normal subgroup of $\G$ of finite index in $W$ and
	let $\H=\No_{W}(\N)$. Then there exists a normal subgroup $\P$ of $\N$ of finite index
	such that $\No_{\H}(\P)=\G$ and $\N/\P$ is a $p$-group. 

  However, the groups $W_n$ for $n \ge 5$ are not virtually free, and
  therefore do not act on a tree in the required way.
\end{remark}

%maybe state the more general theorem
The second half of Problem \ref{prob:main} asks for the cover of $X$ to be polytopal. Subgroups $\P \le W$
such that $\U/\P$ is the flag graph of a polytope are called \emph{semisparse}
(see \cite{HubardToledo2023, Hartley1999}). Hence the question becomes:

\begin{question}
Can the subgroup $\P$ above always be chosen to be semisparse?
\end{question}

\section{Ackowledgements}
The author thanks Primož Potočnik for valuable discussions regarding this work.

This research project was supported by Javna agencija za znanstvenoraziskovalno in inovacijsko
dejavnost Republike Slovenije (ARIS), research program P1-0294. The author acknowledges the use of Claude and ChatGPT for preliminarily checking the author's proof ideas and finding a reference to the proof of \cref{rem:H1}.

\printbibliography
\end{document}